\documentclass[reqno]{amsart}

\usepackage[sorting=none, sorting=nyt, maxbibnames=99, backend=biber]{biblatex}

\renewbibmacro{in:}{}
\bibliography{references.bib}
\usepackage{amssymb}
\usepackage{calrsfs}
\usepackage{graphics,graphicx,mathrsfs}
\usepackage{enumerate}
\usepackage{enumitem}
\usepackage{textcomp}
\usepackage{bbold}
\usepackage{url}
\usepackage{xcolor}
\definecolor{vio}{rgb}{0.54, 0.17, 0.89}
\usepackage[ruled, lined, linesnumbered, longend]{algorithm2e}
\usepackage{hyperref}
\usepackage{titlesec}
\usepackage{todonotes}

\newtheorem{theorem}{Theorem}[section]
\newtheorem{lemma}[theorem]{Lemma}

\newtheorem{proposition}[theorem]{Proposition}

\numberwithin{equation}{section}

\theoremstyle{remark}
\newtheorem*{remark}{Remark}

\titleformat{\section}
  {\normalfont\large\bfseries\centering}{\thesection}{1em}{}

\titleformat{\subsection}
  {\normalfont\bfseries}{\thesubsection}{1em}{}  

\DeclareMathOperator{\Z}{\mathbb{Z}}

\DeclareMathOperator{\R}{\mathbb{R}}

\newcommand{\msP}{\mathscr{P}}

\def\reals{\hbox{\rm I\kern-.18em R}}
\def\complexes{\hbox{\rm C\kern-.43em
\vrule depth 0ex height 1.4ex width .05em\kern.41em}}
\def\field{\hbox{\rm I\kern-.18em F}} 

\let\svthefootnote\thefootnote
\newcommand\freefootnote[1]{%
  \let\thefootnote\relax%
  \footnotetext{#1}%
  \let\thefootnote\svthefootnote%
}

\newenvironment{section*}[2][A]{
  \section*{#2}
  \renewcommand\thesection{#1}
  \setcounter{theorem}{0}}{}

\newcommand{\rev}[1]{\overleftarrow{#1}}

\allowdisplaybreaks

\begin{document}

\title[The infinitude of square-free palindromes]{The infinitude of square-free palindromes}

\author{Daniel R. Johnston}
\thanks{The first author was partly supported by Australian Research Council Discovery Project DP240100186 and an Australian Mathematical Society Lift-off
Fellowship. The second author was supported by Australian Research Council Grants DE220100859 and DP230100534}
\author{Bryce Kerr}
\address{Max Planck Institute for Mathematics, Vivatsgasse 7, 53111 Bonn, Germany}
\email{johnston@mpim-bonn.mpg.de}
\address{School of Science, UNSW Canberra, Northcott Drive, Campbell 2600, Australia}
\email{bryce.kerr@unsw.edu.au}
\date\today

\maketitle 

\begin{small}
\begin{center}
    \textit{Dedicated to Igor Shparlinski on the occasion of his 70th birthday}
\end{center}
\end{small}

\begin{abstract}
    We settle an open problem regarding palindromes; that is, positive integers which are the same when written forwards and backwards. In particular, we prove that for any fixed base $b\geq 2$, there exist infinitely many square-free palindromes in base $b$. We also provide an asymptotic expression for the number of such integers $\leq x$. The core of our proof utilises a hybrid $p$-adic/Archimedean van der Corput process, used in conjunction with an equidistribution estimate of Tuxanidy and Panario, as well as an elementary argument of Cilleruelo, Luca and Shparlinski.
\end{abstract}

\freefootnote{\textit{Corresponding author}: Daniel Johnston (johnston@mpim-bonn.mpg.de).}
\freefootnote{\textit{Affiliation}: Max Planck Institute for Mathematics, Bonn, Germany.}
\freefootnote{\textit{Key phrases}: Palindromes, square-free numbers, van der Corput method, exponential sums.}
\freefootnote{\textit{2020 Mathematics Subject Classification}: 11A63, 11N25 (Primary) 11L07 (Secondary).}

\section{Introduction}
\subsection{Overview}
For a fixed $b\geq 2$, palindromic numbers, or simply palindromes, are positive integers that are the same when written forwards and backwards in base $b$. In base~10, examples include $494$ and $12321$. 

Prior to the 21st century, studies of palindromic numbers were primarily confined to the realm of recreational mathematics. However, in recent years there has been an influx of deep analytic and number-theoretic results regarding palindromes. Examples include the work of Cilleruelo, Luca and Baxter on sums of palindromes \cite{cilleruelo2018every}, or the work of Tuxanidy and Panario \cite{tuxanidy2024infinitude} on almost-prime palindromes. See also ~\cite{luca2003,banks2004,banks2005prime,luca2008binary,col2009palindromes,cilleruelo2009power,cilleruelo2013,raj2020,bai22,zakharov2024} for other modern works on palindromes. Notably, the recent interest in palindromes is in tandem with other studies of the digital properties of numbers. This includes work on primes with missing digits \cite{maynard2019primes,maynard2022primes,leng2025},
with preassigned digits~\cite{bourgain2013,bourgain2015,swaenepoel2020} and reversed primes \cite{bhowmik2024telhcirid,dartyge2024reversible,bhowmik2025zsiflaw,dartyge2025prime}. 

Despite the recent interest in palindromes, little is known about their multiplicative structure. One of the driving reasons for this is that palindromes are so sparse: for a fixed base $b$, there are only about $\sqrt{x}$ palindromes less than $x$. For reference, it is not known whether the similarly sized set
\begin{equation*}
    S(x)=[x,x+\sqrt{x}],
\end{equation*}
always contains a prime for sufficiently large $x$, even under the assumption of the Riemann hypothesis (see e.g.~\cite[\S 1.2]{carneiro2019fourier}).

The problem of whether there exist infinitely many square-free palindromes has often been mentioned in the literature (e.g.~\cite[p.~10]{banks2005prime}, \cite[p.~487]{luca2008binary}, \cite[p.~7]{dartyge2024reversible}, \cite[p.~371]{chourasiya2025power}) and appears as part of Problem 58 in Shparlinski's list of \emph{Problems on Exponential and Character Sums} \cite{Igorlist}. Here, as usual, a square-free number is a positive integer with no repeated prime factors. In this paper we resolve this problem, and in fact give an asymptotic for the number of square-free palindromes satisfying a standard coprimality condition (see Theorem~\ref{asymthm}).
\begin{theorem}\label{shortthm}
    For all $b\geq 2$, there exist infinitely many square-free palindromes in base $b$.
\end{theorem}

In the context of Theorem \ref{shortthm}, it is worth noting that Carillo--Santana \cite[Theorem~1.3]{santana2025}, as well as Chourasiya and the first author \cite[Theorem~1.6]{chourasiya2025power}, recently and independently proved the existence of infinitely many cube-free palindromes. However, in both works, it is clear that their techniques are not strong enough to detect square-free palindromes. In another vein, Tuxanidy and Panario recently showed that there are infinitely many palindromes with at most $6$ prime factors \cite[Theorem 1.4]{tuxanidy2024infinitude}. Although this latter result is very deep, their methods are again not strong enough to detect square-free palindromes. In particular, they do not provide a way to bound the number of palindromes with a large square divisor.

To prove Theorem \ref{shortthm}, we therefore introduce a new approach which combines an elementary argument of Cilleruelo, Luca and Shparlinski \cite{cilleruelo2009power} with the van der Corput method of exponential sums. Throughout our proof of Theorem \ref{shortthm}, we prove a series of exponential sum and Fourier-analytic estimates, and new bounds for the number of palindromes with a large square divisor. These subsidiary results may be of independent interest, particularly in the use of weighted sieve methods applied to sets containing palindromes (see \cite[\S~11]{tuxanidy2024infinitude}). 

Finally we note that Theorem \ref{shortthm} can be viewed as an approximation of the conjecture that there exist infinitely many palindromes which are also prime numbers\footnote{In private communication, C.~Maistret suggested using the term \emph{primendromes} to succinctly refer to such numbers. Alternatively, J. Voight has suggested using the term \emph{palindrimes}.}. A proof of this conjecture still remains firmly out of reach; however, it is known that almost all palindromes are composite~\cite{banks2004}.

\subsection{Main notation and the asymptotic result}
Let $b\geq 2$ be a fixed base. To more precisely define base-$b$ palindromes, consider the base $b$ expansion of an $N$-digit integer $n$
\begin{equation}\label{ndef}
    n=\sum_{0\leq i<N}n_ib^{i}, \quad 0\le n_i<b.
\end{equation}
 If $n_{N-1}\neq 0$ and $n_0\neq 0$, we then define the \emph{digital reverse} of $n$, denoted $\rev{n}$, as
\begin{equation}\label{revndef}
    \rev{n}=\sum_{0\leq i<N}n_i b^{N-1-i}.
\end{equation}
The set of base-$b$ palindromes is then succinctly defined as
\begin{equation*}
    \msP_b=\{n\geq 1: b\nmid n\ \text{and}\ n=\rev{n}\}
\end{equation*}
with the condition $b\nmid n$ encapsulating the need for $n_0\neq 0$ in our expressions \eqref{ndef} and \eqref{revndef}. We also define
\begin{equation*}
    \Pi_b(N):=\msP_b\cap[b^{N-1},b^{N})\quad\text{and}\quad \msP_b(x):=\{n\in\msP_b:n\leq x\}
\end{equation*}
for the set of $N$-digit palindromes and the set of palindromes $\leq x$ respectively. 

When studying palindromes, it is common to restrict to palindromes $n$ with $(n,b^3-b)=1$. This is to avoid technicalities owing to relations between $n$ and $\rev{n}$ mod $b^3-b$; see the introduction of \cite{bhowmik2024telhcirid} for further discussion on this matter. Consequently, we will work with the restricted sets
\begin{align}
    \Pi_b^*(N)&:=\{n\in\Pi_b(N):(n,b^3-b)=1\},\notag\\
    \msP_b^*(x)&:=\{n\in\msP_b(x):(n,b^3-b)=1\}.\label{pstardef}
\end{align}
Since an $N$-digit palindrome is completely determined by its first $\lceil N/2\rceil$ digits, a simple combinatorial argument yields that
\begin{equation*}
    \#\Pi_b(N)\asymp_b b^{N/2}\quad\text{and}\quad \#\msP_b(x)\asymp_b\sqrt{x}.
\end{equation*}
Similarly, one can show (see \cite[Lemma 9.1]{tuxanidy2024infinitude})
\begin{equation}\label{pistarasym}
    \#\Pi_b^*(N)\asymp_b
    \begin{cases}
        b^{N/2},&\text{if $N$ is odd},\\
        0,&\text{if $N$ is even},
    \end{cases}
\end{equation}
and
\begin{equation}\label{starasym}
     \quad\#\msP_b^*(x)\asymp_b\sqrt{x}.
\end{equation}
Note that in~\eqref{pistarasym}, the parity of $N$ is important since every even digit palindrome is divisible by $b+1$. This follows upon applying the congruence $b\equiv-1$ (mod $b+1$) to the base-$b$ expansion of an even digit palindrome. Now, with the above notation, an asymptotic version of our result (Theorem \ref{shortthm}) is given as follows.
\begin{theorem}\label{asymthm}
    For all bases $b\geq 2$, if
    \begin{equation*}
        Q_b^*(x):=\#\{n\in\msP_b^*(x): n\ \text{is square-free}\}
    \end{equation*}
    then one has, for any $A>0$,
    \begin{equation}\label{Qbexp}
        Q_b^*(x)=\frac{\#\msP_b^*(x)}{\zeta(2)}\prod_{p\mid b^3-b}\left(1-\frac{1}{p^2}\right)^{-1}\left(1+O_{b,A}\left(\frac{1}{(\log x)^A}\right)\right),
    \end{equation}
    where $\zeta(\cdot)$ is the Riemann zeta function.
\end{theorem}
Notably, the expression in \eqref{Qbexp} has the same form as the asymptotic expression for $k$-free palindromes ($k\geq 3$) given in \cite[Theorem~1.6]{chourasiya2025power}. Here, 
\begin{equation*}
    \frac{1}{\zeta(2)}\prod_{p\mid b^3-b}\left(1-\frac{1}{p^2}\right)^{-1}=\frac{6}{\pi^2}\prod_{p\mid b^3-b}\left(1-\frac{1}{p^2}\right)^{-1}
\end{equation*}
is the relative density of square-free numbers coprime to $b^3-b$. Therefore, \eqref{Qbexp} essentially says that after restricting to integers $n$ with $(n,b^3-b)=1$, the probability of $n$ being a palindrome is independent of the probability of $n$ being square-free. 

The behaviour of square-free palindromes in the unrestricted set $\mathcal{P}_b(x)$ appears to be much more chaotic. However, if we fix the digital length $N$ and consider
\begin{equation*}
    \mathcal{Q}_b(N):=\#\{n\in\Pi_b(N):n\ \text{is square-free}\}
\end{equation*}
for the number of square-free elements of $\Pi_b(N)$ instead of $\msP_b^*(x)$, then numerical computations suggest that
\begin{equation}\label{fullQbeq}
    \mathcal{Q}_b(N)\sim\frac{\#\Pi_b(N)}{\zeta(2)}.
\end{equation}
Whilst the expression \eqref{fullQbeq} is very elegant, we do not attempt to prove it here. This is because working with $\msP_b^*(x)$ or $\Pi_b^*(N)$ rather than $\msP_b(x)$ or $\Pi_b(N)$ allows us to simplify many arguments. However, we believe that it should be possible to prove \eqref{fullQbeq} by a routine reworking of our approach and the results we rely on from the literature.  

\subsection{Structure of paper}
A summary of the rest of the paper is as follows. In Section \ref{Outlinesect} we give an outline of the proof of Theorem \ref{asymthm} and thus Theorem \ref{shortthm}. More precisely, we state our main subsidiary results (Propositions~\ref{powerprop1}, \ref{powerprop2} and \ref{powerprop3}) and show how these imply Theorem~\ref{asymthm}. In Section~\ref{Prop1sect} we then prove the first of these subsidiary results, Proposition~\ref{powerprop1}, which gives an elementary bound on the number of palindromes with a large square divisor. Next, in Section~\ref{Prelimsect} we discuss and prove all of the preliminaries required for our variant of the van der Corput method. Finally, we conclude in Sections~\ref{Prop2sect} and \ref{Prop3sect} by giving the proofs of Propositions~\ref{powerprop2} and \ref{powerprop3}. 

Throughout, we employ Vinogradov's notation $f\ll g$ to mean $f=O(g)$. A ``smooth" function will always refer to a function $f$ such that the $k^{\text{th}}$ derivative $f^{(k)}$ exists and is continuous for all $k\geq 1$. We also use the standard notation $\mu(d)$ for the M\"obius function, $\tau(n)$ for the divisor function, $e(x):=\exp(2\pi ix)$ for complex exponentials, and $d\sim D$ to mean that $d\in[D,2D]$. This latter notation should not be confused with the identical notation $f(x)\sim g(x)$ (see e.g.~\eqref{fullQbeq}) reserved for an asymptotic equivalence of functions.  

\section*{Acknowledgements}
DRJ is grateful to the Max Planck Institute for Mathematics in Bonn
for its hospitality and financial support. We thank Aleksandr Tuxanidy for the collegiate and insightful discussions. Notably, Aleksandr informed us that he was working independently on the problem of infinitely many square-free palindromes, and has since released a preprint~\cite{tuxanidy2026distribution} on the topic containing a quantitatively stronger version of Theorem~\ref{powerthm}. We also thank Igor Shparlinski for some minor remarks on the first version of this manuscript.

\section{Outline of the proof of Theorem~\ref{asymthm}}\label{Outlinesect}
To set up the proof of Theorem \ref{asymthm}, we first use M\"obius inversion to write
\begin{equation}\label{mobeq}
    Q_b^*(x)=\sum_{n\in\msP_b^*(x)}\mu^2(n)=\sum_{n\in\msP_b^*(x)}\sum_{d^2\mid n}\mu(d)=\sum_{\substack{d\leq \sqrt{x}\\ (d,b^3-b)=1}}\sum_{\substack{n\in\msP_b^*(x)\\ d^2\mid n}}\mu(d).
\end{equation}
Our underlying approach is then to split the sum in \eqref{mobeq} into different ranges for $d$ and apply relevant estimates. A formal proof of how these estimates imply Theorem~\ref{asymthm} is given at the end of this section. The first estimate we use is the following equidistribution result for square moduli, proven by Tuxanidy and Panario \cite{tuxanidy2024infinitude} as an application of a large sieve inequality of Baier and Zhao \cite{baier2008improvement}. 
\begin{proposition}[{\cite[Proposition 10.1]{tuxanidy2024infinitude}}]\label{equiprop1}
    For any $b\geq 2$, $\varepsilon>0$ and $A>0$, one has
    \begin{equation}\label{equieq1}
        \sum_{\substack{d\leq x^{1/4-\varepsilon}\\(d,b^3-b)=1}}\mu^2(d)\sup_{y\leq x}\max_{a\in\mathbb{Z}}\left|\sum_{n\in\msP_b^*(y)}\left(\mathbb{1}_{n\equiv a \mathrm{(}\mathrm{mod}\  d^2\mathrm{)}}-\frac{1}{d^2}\right)\right|\ll_{A,b,\varepsilon}\frac{\#\msP_b^*(x)}{(\log x)^A}.
    \end{equation}
\end{proposition}
Proposition \ref{equiprop1} allows us to treat the range $1\leq d\leq x^{1/4-\varepsilon}$ in \eqref{mobeq}, from which we can extract the main term of $Q_b^*(x)$. After this, we need sufficiently strong upper bounds for the number of palindromes which have a large square divisor $d^2$, with $d>x^{1/4-\varepsilon}$. This will be done by considering the function
\begin{equation*}
    S_b(x,D):=\#\{n\in\msP^*_b(x):\exists d\sim D\ \text{s.t.}\ d^2\mid n\}
\end{equation*}
which counts the number of $n\in\msP_b^*(x)$ divisible by $d^2$ with $d$ in a dyadic interval $[D,2D]$. The first bound we obtain for $S_b(x,D)$ is given below. Recalling that $\#\msP_b^*(x)\asymp_b\sqrt{x}$ (see \eqref{starasym}), we find that this result is non-trivial in the range $x^{1/3+\varepsilon}\leq D\leq \sqrt{x}$.
\begin{proposition}\label{powerprop1}
    For all $b\geq 2$ and $D\geq 1$, we have
    \begin{equation*}
         S_b(x,D)\ll_b\frac{x}{D^{3/2}}.
    \end{equation*}
    In particular, for any $\varepsilon>0$, if $D\geq x^{1/3+\varepsilon}$ then
    \begin{equation*}
         S_b(x,D)\ll_b x^{1/2-\varepsilon}.
    \end{equation*}
\end{proposition}

The proof of Proposition~\ref{powerprop1} is relatively simple, being an elaboration of an elementary argument due to Cilleruelo, Luca and Shparlinski~\cite[Theorem 1]{cilleruelo2009power} which reduces counting palindromes with a large square divisor to counting integers in short intervals and arithmetic progressions.

Given Proposition \ref{powerprop1}, it remains to cover the range $x^{1/4-\varepsilon}<D<x^{1/3+\varepsilon}$. To do so, we incorporate Fourier-analytic and exponential sum methods into the proof of Proposition \ref{powerprop1}, yielding bounds for $S_b(x,D)$ that are better for smaller values of $D$. The overarching approach here will be to apply a van der Corput method of exponential sums. For background on this method see \cite{graham1991van}. An explanation of our modification of this method is then given in Section \ref{Prelimsect}. Ultimately, we are able to obtain the following two bounds, which are non-trivial for different ranges of $D$. 
\begin{proposition}
\label{powerprop2}
    For all $b\geq 2$ and $\varepsilon>0$, if $x^{1/4}\leq D\leq x^{2/5}$ then
    \begin{equation}\label{57eq}
        S_b(x,D)\ll_{b,\varepsilon}\frac{x^{2/3}}{D^{2/3-\varepsilon}}.
    \end{equation}
    In particular, for any $\varepsilon>0$, if  $x^{1/4+3\varepsilon}\leq D\leq x^{2/5}$ then
    \begin{equation*}
       S_b(x,D)\ll_{b,\varepsilon} x^{1/2-\varepsilon}.
    \end{equation*}    
\end{proposition}
\begin{proposition}
\label{powerprop3}
    For all $b\geq 2$ and $\varepsilon>0$, if $x^{2/9}\leq D\leq x^{8/31-\varepsilon}$ then
    \begin{equation}\label{58eq}
        S_b(x,D) \ll_{b,\varepsilon} \frac{x^{3/5}}{D^{9/20-\varepsilon}}.
    \end{equation}
    In particular, for any $\varepsilon>0$, if  $x^{2/9+5\varepsilon}\leq D\leq x^{8/31-\varepsilon}$ then
    \begin{equation*}
        S_b(x,D)\ll_{b,\varepsilon} x^{1/2-\varepsilon}.
    \end{equation*}    
\end{proposition}

Combining Propositions \ref{powerprop1}, \ref{powerprop2} and \ref{powerprop3} yields the following theorem.
\begin{theorem}\label{powerthm}
    For all $b\geq 2$ and $\varepsilon>0$, if $D\geq x^{2/9+\varepsilon}$ then there exists a $\delta>0$ such that
    \begin{equation*}
        S_b(x,D)\ll_{b,\varepsilon}x^{1/2-\delta}.
    \end{equation*}
\end{theorem}
Since $2/9<1/4$, Theorem \ref{powerthm} is sufficient for our purposes.  

We conclude this section by formalising the above discussion, giving an explicit proof as to how Theorem \ref{asymthm} follows from Proposition~\ref{equiprop1} and Theorem~\ref{powerthm}. The rest of the paper is then concerned with the proofs of Propositions~\ref{powerprop1}, \ref{powerprop2} and~\ref{powerprop3}.
\begin{proof}[Proof of Theorem~\ref{asymthm} assuming Proposition~\ref{equiprop1} and Theorem \ref{powerthm}.]
    Beginning with \eqref{mobeq}, we write
    \begin{equation*}
        Q_b^*(x)=\sum_{\substack{d\leq \sqrt{x}\\ (d,b^3-b)=1}}\sum_{\substack{n\in\msP_b^*(x)\\ d^2\mid n}}\mu(d)=S_1(x)+S_2(x),
    \end{equation*}
    where
    \begin{align*}
        S_1(x)&:=\sum_{\substack{d\leq x^{0.24}\\ (d,b^3-b)=1}}\sum_{\substack{n\in\msP_b^*(x)\\ d^2\mid n}}\mu(d),\\
        S_2(x)&:=\sum_{\substack{x^{0.24}<d\leq\sqrt{x}\\ (d,b^3-b)=1}}\:\sum_{\substack{n\in\msP_b^*(x)\\ d^2\mid n}}\mu(d).
    \end{align*}
    We first deal with $S_1(x)$. Here, Proposition \ref{equiprop1} implies that
    \begin{align}
        S_1(x)&=\#\msP_b^*(x)\sum_{\substack{d\leq x^{0.24}\\(d,b^3-b)=1}}\frac{\mu(d)}{d^2}+O_{A,b}\left(\frac{\#\msP_b^*(x)}{(\log x)^A}\right)\notag\\
        &=\#\msP_b^*(x)\left(\sum_{\substack{d\geq 1\\(d,b^3-b)=1}}\frac{\mu(d)}{d^2}-\sum_{\substack{d>x^{0.24}\\(d,b^3-b)=1}}\frac{\mu(d)}{d^2}\right)+O_{A,b}\left(\frac{\#\msP_b^*(x)}{(\log x)^A}\right).\label{s1eq1}
    \end{align}
    Converting the first sum in \eqref{s1eq1} into an Euler product gives
    \begin{equation}\label{s1eq2}
        \sum_{\substack{d\geq 1\\(d,b^3-b)=1}}\frac{\mu(d)}{d^2}=\prod_{\substack{p\geq 2\\(p,b^3-b)=1}}\left(1-\frac{1}{p^2}\right)=\frac{1}{\zeta(2)}\prod_{p\mid b^3-b}\left(1-\frac{1}{p^2}\right)^{-1}.
    \end{equation}
    Then,
    \begin{equation}\label{s1eq3}
        \sum_{\substack{d>x^{0.24}\\(d,b^3-b)=1}}\frac{\mu(d)}{d^2}\ll\sum_{d>x^{0.24}}\frac{1}{d^2}\ll\frac{1}{x^{0.24}}.
    \end{equation}
    Hence, upon substituting \eqref{s1eq2} and \eqref{s1eq3} into \eqref{s1eq1}, one has
    \begin{equation*}
        S_1(x)= \frac{\#\msP_b^*(x)}{\zeta(2)}\prod_{p\mid b^3-b}\left(1-\frac{1}{p^2}\right)^{-1}\left(1+O_{A,b}\left(\frac{1}{(\log x)^A}\right)\right).
    \end{equation*}
    This is the desired asymptotic for $Q_b^*(x)$ in Theorem~\ref{asymthm}. Therefore, it suffices to show that $S_2(x)$ is of a smaller order. With this aim, we apply the triangle inequality and split into dyadic intervals to yield
    \begin{align*}
        |S_2(x)|\leq\sum_{x^{0.24}\leq d\leq \sqrt{x}}\:\sum_{\substack{n\in\msP_b^*(x)\\ d^2\mid n}}1 &\leq \sum_{m=\lfloor\log_2 x^{0.24}\rfloor}^{\lceil\log_2\sqrt{x}\rceil}\:\sum_{d\sim 2^m}\:\sum_{\substack{n\in\msP_b^*(x)\\ d^2\mid n}}1 \\ 
        &\leq \sum_{m=\lfloor\log_2 x^{0.24}\rfloor}^{\lceil\log_2\sqrt{x}\rceil}\:\sum_{\substack{n\in\msP_b^*(x)\\ \exists d\sim 2^m,\:d^2\mid n}}\tau(m),
    \end{align*}
    where we have used the standard notation $\log_2 x:=\log x/\log 2$. Now, applying Theorem~\ref{powerthm} and the bound $\tau(n)\ll_{\varepsilon} n^{\varepsilon}$ for any fixed $\varepsilon>0$, we have that there exists $\delta>0$ such that
    \begin{align*}
        S_2(x)&\ll_b x^{\delta/2}\sum_{m=\lfloor\log_2 x^{0.24}\rfloor}^{\lceil\log_2\sqrt{x}\rceil}S_b(x,2^{m})\\
        &\ll_b\sum_{m=\lfloor\log_2 x^{0.24}\rfloor}^{\lceil\log_2\sqrt{x}\rceil} x^{\frac{1}{2}-\frac{\delta}{2}}\ll x^{\frac{1}{2}-\frac{\delta}{2}}\log x=o_{b,A}\left(\frac{\#\msP_b^*(x)}{(\log x)^A}\right)
    \end{align*}
    as desired.
\end{proof}

\section{Proof of Proposition~\ref{powerprop1}}\label{Prop1sect}
In this section we prove Proposition \ref{powerprop1}. The proof is elementary and inspired by that of \cite[Theorem 1]{cilleruelo2009power}. In particular, the simplest case of \cite[Theorem 1]{cilleruelo2009power} gives an upper bound for the number of \emph{square} palindromes and we adapt their proof to obtain an upper bound for the number of palindromes with a large \emph{square divisor}.

To begin with, we give a couple of simple lemmas.
\begin{lemma}\label{sqrtlem}
    Suppose for some $d,X,Y\geq 1$ that
    \begin{equation*}
        |d^2-Y|\leq X.
    \end{equation*}
    Then,
    \begin{equation*}
        |d-\sqrt{Y}|\leq\frac{X}{\sqrt{Y}}.
    \end{equation*}
\end{lemma}
\begin{proof}
Since $d,Y\geq 1$, we have 
\begin{align*}
\sqrt{Y}|d-\sqrt{Y}|\le |(d+\sqrt{Y})(d-\sqrt{Y})|=|d^2-Y|\le X
\end{align*}
from which the result follows.
\end{proof}
\begin{lemma}\label{lem:cubres}
    Let $a\in\mathbb{Z}$ and $q\geq 2$. Then, the number of solutions $w$ mod $q$ with $(w,q)=1$ to the congruence
    \begin{equation}\label{w3eq}
        w^k\equiv a\pmod{q}
    \end{equation}
    is at most $2k^{\omega(q)}$, where $\omega(q)$ is the number of distinct prime factors of $q$.
\end{lemma}

\begin{proof} 
    The case where $q$ is a prime power is, for example, covered in Section 4.2 of the textbook \cite{ireland1990classical}. Here, one obtains that \eqref{w3eq} has at most $k$ solutions if $q$ is a power of an odd prime, and at most $2k$ solutions if $q$ is a power of $2$. For general $q$, one then applies the Chinese remainder theorem.
\end{proof}

We now prove Proposition \ref{powerprop1}.

\begin{proof}[Proof of Proposition \ref{powerprop1}]
    We begin by proving the result for palindromes of fixed digit length $N$. That is, we show
    \begin{equation*}
        \#\{n\in\Pi_b^*(N):\exists d\sim D\ \text{s.t.}\ d^2\mid n\}\ll_b\frac{b^N}{D^{3/2}}.
    \end{equation*}
    Let $K$ be the largest integer such that $b^K\leq\sqrt{D}$. This implies that
    \begin{equation}\label{bKeq}
        b^K\asymp_b\sqrt{D}.
    \end{equation}
    Now, there exists an integer $a$ with $0\leq a<b^K$ such that
    \begin{align}\label{acongeq}
        &\#\{n\in\Pi_b^*(N):\exists d\sim D\ \text{s.t.}\ d^2\mid n\}\notag\\
        &\qquad\qquad\leq b^K\#\{n\in\Pi_b^*(N):\exists d\sim D\ \text{s.t.}\ d^2\mid n,\ n\equiv a\ \text{(mod}\ b^K\text{)}\}.
    \end{align}
    Let $n\in\Pi_b^*(N)$ be a palindrome counted in \eqref{acongeq}. The condition $n\equiv a\pmod{b^K}$ indicates that the least significant $K$ digits of $n$ are equal to $a$. Since the last base-$b$ digit of a palindrome cannot be $0$, we may assume that $a$ does not end in a $0$. Now, let $\alpha$ be the $N$-digit number with its leading digits the mirror image of $a$ followed by trailing zeros. So, in particular, $\alpha$ determines the $K$ most significant digits of the palindrome $n$. We thus have
    \begin{equation*}
        \left|n-\alpha \right|\leq \frac{b^N}{b^K}.
    \end{equation*}
    Therefore,
    \begin{align*}
        &\#\{n\in\Pi_b^*(N):\exists d\sim D\ \text{s.t.}\ d^2\mid n\}\\
        &\leq b^K\#\{n\in\Pi_b^*(N):\exists d\sim D\ \text{s.t.}\ d^2\mid n,\ n\equiv a\ \text{(mod}\ b^K\text{)}\}\\
        &\leq b^K\#\left\{d\sim D,\ s\in\left[\frac{b^{N-1}}{4D^2},\frac{b^N}{D^2}\right]: d^2s\equiv a\ \text{(mod}\ b^K\text{)},\ (sd,b)=1,\ \left|d^2s-\alpha \right|\leq \frac{b^{N}}{b^K}\right\}
    \end{align*} 
    where we have the condition $(sd,b)=1$ since any $n\in\Pi_b^*(N)$ is coprime to $b$. From here, we deduce that there exists $s\in[b^{N-1}/4D^2,b^N/D^2]$ with $(s,b)=1$ and
    \begin{align}\label{fixseq}
        &\#\{n\in\Pi_b^*(N):\exists d\sim D\ \text{s.t.}\ d^2\mid n\}\notag\\
        &\quad\ll \frac{b^{N+K}}{D^2}\#\left\{d\sim D: d^2\equiv as^{-1}\  \text{(mod}\ b^K\text{)},\ (d,b)=1,\ \left|d^2s-\alpha\right|\leq \frac{b^{N}}{b^K}\right\}.
    \end{align}
    By Lemma \ref{sqrtlem}, we have
    \begin{equation}\label{dsqrteq}
        \left|d^2s-\alpha\right|\leq \frac{b^{N}}{b^K}\Rightarrow \left|d-\sqrt{\frac{\alpha}{s}}\right|\ll_b\frac{D}{b^K}
    \end{equation}
    where we have used $\alpha\asymp_b b^N$ and $s\asymp_b b^N/D^2$. Now, since $(d,b)=1$, Lemma~\ref{lem:cubres} implies that the number of $d$ in an interval $[X,Y]$ with $d^2\equiv as^{-1}$ (mod $b^K$) is ${\ll_b (Y-X)/b^K+1}$. Hence, the number of $d$ satisfying \eqref{dsqrteq} and ${d^2\equiv as^{-1}\ \text{(mod $b^K$)}}$ is $\ll_b D/b^{2K}+1$. Substituting this information into \eqref{fixseq} and applying \eqref{bKeq} yields
    \begin{equation*}
        \#\{n\in\Pi_b^*(N):\exists d\sim D\ \text{s.t.}\ d^2\mid n\}\ll_b\frac{b^N}{Db^K}+\frac{b^{N+K}}{D^2}\ll_b\frac{b^N}{D^{3/2}}
    \end{equation*}
    as desired. To convert this to a result over all $n\in\msP_b^*(x)$, we simply note that
    \begin{align*}
        \#\{n\in\msP_b^*(x):\exists d\sim D\ \text{s.t.}\ d^2\mid n\}&\leq\sum_{N=1}^{\lceil\log x/\log b\rceil}\#\{n\in\Pi_b^*(N):\exists d\sim D\ \text{s.t.}\ d^2\mid n\}\\
        &\ll_b\sum_{N=1}^{\lceil\log x/\log b\rceil}\frac{b^N}{D^{3/2}}\ll_b\frac{x}{D^{3/2}},
    \end{align*}
    as required. 
\end{proof}

\section{Preliminaries for the van der Corput method}\label{Prelimsect}

To prove Propositions~\ref{powerprop2} and~\ref{powerprop3}, we modify the proof of Proposition \ref{powerprop1} by applying a hybrid $p$-adic and Archimedean van der Corput method to detect integers in arithmetic progressions and short intervals. This method uses two processes:
\begin{itemize}
    \item An $A$-process which is based on Weyl differencing (see Lemma \ref{lem:vdc}).
    \item A $B$-process which is based on Poisson summation (see Lemma \ref{lem:ps}).
\end{itemize}
The proof of Proposition~\ref{powerprop2} uses only the $B$-process which sets the ``convexity barrier" ${D\ge x^{1/4+\varepsilon}}$. Then, Proposition~\ref{powerprop3} uses a combination of the $B$-process and the $A$-process to treat slightly smaller values of $D$ and break the $x^{1/4}$-barrier. One could certainly try to elaborate on our approach to obtain non-trivial results for even smaller $D$, however this is unnecessary for our application.

Throughout the rest of this section, we provide all of the preliminaries required for this van der Corput method. Namely, we state the $A$ and $B$ processes we use, and prove a series of relevant results on oscillatory integrals and exponential sums.

\subsection{The $A$ and $B$ processes}
The $A$-process we use is a smooth variation of the usual Weyl-van der Corput inequality. The proof is similar to that of the unsmoothed inequality (see e.g.\ \cite[Lemma 2.5]{graham1991van}). However, we include all the details here for clarity.
\begin{lemma} 
\label{lem:vdc}
Let $z_n$ be a sequence of complex numbers satisfying $|z_n|=O(1)$, and $\psi$ a smooth function with support in $[1/2,5/2]$ and $\psi(x)=1$ for ${x\in[1,2]}$. Then for any $Q\le D^{1/2}$,
\begin{align*}
\left|\sum_{d\sim D}z_d\right|\ll \frac{D}{Q^{1/2}}+\frac{D^{1/2}}{Q^{1/2}}\left(\sum_{q\le Q}\left|\sum_{d\in \Z}\psi\left(\frac{d}{D}\right)z_d\overline{z}_{d+q}\right|\right)^{1/2}.
\end{align*}
\end{lemma}
\begin{proof}
Since for any $q\le Q$
\begin{align*}
\sum_{d\sim D}z_d=\sum_{d+q\sim D}z_{d+q}=\sum_{d\sim D}z_{d+q}+O(Q)
\end{align*}
we have 
\begin{align*}
\sum_{d\sim D}z_d&=\frac{1}{Q}\sum_{\substack{1\le q \le Q}}\sum_{\substack{d\sim D}}z_{d+q}+O(Q) \ll \frac{1}{Q}\sum_{d\sim D}\left|\sum_{\substack{1\le q \le Q }}z_{d+q}\right|+Q.
\end{align*}
Hence, by the Cauchy-Schwarz inequality,
\begin{align}\label{weyl1eq}
    \left|\sum_{d\sim D}z_d\right|^2&\ll \frac{D}{Q^2}\sum_{d\sim D}\left|\sum_{\substack{1\le q \le Q}}z_{d+q}\right|^2+Q^2\notag\\
    &\ll \frac{D}{Q^2}\sum_{d\in \Z}\psi\left(\frac{d}{D}\right)\left|\sum_{\substack{1\le q \le Q }}z_{d+q}\right|^2+Q^2
\end{align}
with $\psi$ as in the statement of the lemma. Expanding the square in \eqref{weyl1eq} gives
\begin{equation}\label{weyleq2}
    \left|\sum_{d\sim D}z_d\right|^2\ll \frac{D^2}{Q}+\frac{D}{Q^2}\sum_{\substack{1\le q_1<q_2 \le Q}}\left|\sum_{\substack{d\in \Z }}\psi\left(\frac{d}{D}\right)z_{d+q_1}\overline{z}_{d+q_2} \right|+Q^2,    
\end{equation}
with the $D^2/Q$ term coming from removing the diagonal contribution with $q_1=q_2$. Next, we perform a change of variable $d\to d-q_1$ so that
\begin{align*}
    \left|\sum_{d\sim D}z_d\right|^2\ll \frac{D^2}{Q}+Q^2+\frac{D}{Q^2}\sum_{\substack{1\le q_1<q_2 \le Q}}\left|\sum_{\substack{d\in \Z }}\psi\left(\frac{d-q_1}{D}\right)z_{d}\overline{z}_{d+(q_2-q_1)} \right|.
\end{align*}
By the mean value theorem  $$\psi\left(d-q_1/D\right)= \psi(d/D)+O(Q/D).$$ Hence, since $Q\le D^{1/2}$,
\begin{align}
    \left|\sum_{d\sim D}z_d\right|^2&\ll \frac{D^2}{Q}+DQ+\frac{D}{Q^2}\sum_{1\leq q_1<q_2\leq Q}\left|\sum_{\substack{d\in \Z }}\psi\left(\frac{d}{D}\right)z_{d}\overline{z}_{d+(q_2-q_1)} \right|\notag\\
    &\ll\frac{D^2}{Q}+\frac{D}{Q^2}\sum_{1\leq q_1<q_2\leq Q}\left|\sum_{\substack{d\in \Z }}\psi\left(\frac{d}{D}\right)z_{d}\overline{z}_{d+(q_2-q_1)} \right|,\label{weyleq3}
\end{align}
where in the second line we have used that $DQ\leq D^2/Q$ since $Q\leq D^{1/2}$. Finally, we write $q=q_2-q_1$ and note that each difference $q_2-q_1$ occurs at most $Q$ times in \eqref{weyleq3}, and the desired result follows.
\end{proof}

The $B$-process is then given by the following variation of Poisson summation which, for example, may be readily derived from~\cite[Lemma 2.9]{bombieri1986order}. 
\begin{lemma}
\label{lem:ps}
Let $f\in L^{1}(\R)$ be a continuous, compactly supported function, and $g$ be a periodic function mod $q$. Then 
\begin{align}\label{poissoneq}
\sum_{n\in \Z}f(n)g(n)=\frac{1}{q^{1/2}}\sum_{m\in \Z}\widehat{f}\left(\frac{m}{q}\right)\widehat{g}(m)
\end{align}
where $\widehat{f}$ denotes the Fourier transform of $f$ over $\R$:
\begin{equation}\label{Foureq}
    \widehat{f}(k)=\int_{\mathbb{R}}f(u)e(-ku)\mathrm{d}u    
\end{equation}
and $\widehat{g}$ denotes the Fourier transform of $g$ in $\Z/q\Z$:
\begin{align*}
\widehat{g}(m)=\frac{1}{q^{1/2}}\sum_{y=1}^{q}g(y)e\left(\frac{my}{q}\right).
\end{align*}
\end{lemma}

The presence of $\widehat{f}(m/q)$ in \eqref{poissoneq} is the ``Archimedean" component, requiring estimates for oscillatory integrals over $\R$. On the other hand, the presence of $\widehat{g}(m)$ requires estimates for ``$p$-adic" oscillatory integrals, namely exponential sums.

\subsection{Oscillatory integral estimates}\label{osssub}

We now provide a series of lemmas related to estimating oscillatory integrals over $\R$. The following result was proven by Heath-Brown~\cite{heath1988growth} via the repeated use of integration by parts.
\begin{lemma}[{\cite[Lemma 4]{heath1988growth}}]
\label{lem:nostat}
Let $W$ and $\phi$ be smooth functions on $[a,b]$. Let 
$$\Phi=\inf\{ |\phi'(x)| \ : x\in [a,b] \}.$$
Suppose further that 
$$W^{(k)}(x) \ll (b-a)^{-k}$$
and 
$$\phi^{(k)}(x)\ll (b-a)^{1-k}\Phi.$$
Then for any $N\ge 1$
\begin{align*}
\int_{a}^{b}e^{i\phi(x)}W(x)\mathrm{d}x\ll (b-a)^{1-N}\Phi^{-N}.
\end{align*}
\end{lemma}

We also require some results regarding the mean value of definite integrals. We begin with a statement of the ``second mean value theorem". See \cite{hobson1909second} for further discussion and a proof of this classical result.

\begin{lemma}[Second mean value theorem]\label{secondmvlem}
    Let $g:[a,b]\to\R$ be integrable. If $f:[a,b]\to\R$ is a positive monotonically decreasing function, then there exists $c\in(a,b]$ such that
    \begin{equation*}
        \int_a^bf(x)g(x)\mathrm{d}x=f(a^+)\int_a^cg(x)\mathrm{d}x,
    \end{equation*}
    where $f(a^+)=\lim_{x\to a^+}f(x)$. On the other hand, if $f:[a,b]\to\R$ is a positive monotonically increasing function, then there exists $c\in[a,b)$ such that
    \begin{equation*}
        \int_a^bf(x)g(x)\mathrm{d}x=f(b^-)\int_c^bg(x)\mathrm{d}x,
    \end{equation*}
    where $f(b^-)=\lim_{x\to b^-}f(x)$.
\end{lemma}
Using the second mean value theorem, we state two well known estimates on oscillatory integrals. Similar results are given in the literature (e.g.\ \cite[Lemmas~4.3 and 4.5]{titchmarsh1986theory}); however they are not in the right form for our purposes so we provide complete details.

\begin{lemma}\label{mvlem}
    Let $F(x)$ be a differentiable function on $[a,b]$ such that $F'$ is monotonic with ${F'(x)>m>0}$ or $F'(x)<-m<0$. Suppose also that $G(x)$ is a monotonic function on $[a,b]$ with $0\leq G(x)\leq M$ for some constant $M>0$. Then,
    \begin{equation}\label{mveq1}
        \left|\int_a^bG(x)e^{iF(x)}\mathrm{d}x\right|\leq\frac{4M}{m}.
    \end{equation}
\end{lemma}
\begin{proof}
    First we consider the real part of the integral
    \begin{equation}\label{FGcoseq}
        \int_a^bG(x)\cos(F(x))\mathrm{d}x.
    \end{equation}
    Throughout we assume that $F'(x)$ and $G(x)$ are monotonically increasing on $[a,b]$, since the argument is essentially identical for the other cases of monotonicity. By the second mean value theorem (Lemma \ref{secondmvlem}), there exists $c\in[a,b)$ such that
    \begin{equation*}
        \int_{a}^{b}G(x)\cos(F(x))\mathrm{d}x=G(b)\int_{c}^{b}\cos(F(x))\mathrm{d}x.
    \end{equation*}
    If $F'(x)>m>0$, then applying the second mean value theorem again, we have that there exists $d\in(c,b]$ such that
    \begin{align*}
        G(b)\int_{c}^{b}\cos(F(x))\mathrm{d}x&=\frac{G(b)}{F'(c)}\int_{c}^{d}F'(x)\cos(F(x))\mathrm{d}x\\
        &=\frac{G(b)}{F'(c)}\left(\sin\{F(d)\}-\sin\{F(c)\}\right).
    \end{align*}
    Otherwise, if $F'(x)<-m<0$, then $-1/F'(x)$ is positive and increasing and there exists $e\in[c,b)$ such that
    \begin{align*}
        G(b)\int_{c}^{b}\cos(F(x))\mathrm{d}x&=-\frac{G(b)}{F'(b)}\int_{e}^{b}F'(x)\cos(F(x))\mathrm{d}x\\
        &=-\frac{G(b)}{F'(b)}\left(\sin\{F(b)\}-\sin\{F(e)\}\right).
    \end{align*}
    In either case,
    \begin{equation*}
        \left|\int_{a}^{b}G(x)\cos(F(x))\mathrm{d}x\right|=\left|G(b)\int_{c}^{b}\cos(F(x))\mathrm{d}x\right|\leq\frac{2M}{m}.
    \end{equation*}
    The same argument holds for the imaginary part of the integral in \eqref{mveq1}, namely,
    \begin{equation*}
        \left|\int_{a}^{b}G(x)\sin(F(x))\mathrm{d}x\right|\leq\frac{2M}{m}.
    \end{equation*}
    Hence, adding the bounds for real and imaginary parts together we obtain the desired result.
\end{proof}

\begin{lemma}\label{mvlem2}
    Let $F(x)$ be a real, twice differentiable function with either ${F''(x)>r>0}$ or $F''(x)<-r<0$ throughout the interval $[a,b]$. Suppose also that $G(x)$ is a real piecewise monotonic function on the interval $[a,b]$ with $0\leq G(x)\leq M$ for some constant $M>0$. Then,
    \begin{equation}\label{mveq2}
        \left|\int_a^bG(x)e^{iF(x)}\mathrm{d}x\right|\leq\frac{8KM}{\sqrt{r}}\ll_{K,M}\frac{1}{\sqrt{r}},
    \end{equation}
    where $K$ is the number of piecewise monotonic components of $G$.
\end{lemma}
\begin{proof}
    We begin by making a finite subdivision of the integral in \eqref{mveq2} as
    \begin{equation}\label{akbkeq2}
        \int_a^bG(x)e^{iF(x)}\mathrm{d}x=\sum_{k=1}^K\int_{a_k}^{b_k}G(x)e^{iF(x)}\mathrm{d}x
    \end{equation}
    so that $G(x)$ is monotonic on each $[a_k,b_k]$ except possibly at the endpoints.
    
    Fix $k\in\{1,\ldots,K\}$. From here onwards we assume $G(x)$ is monotonically increasing on $[a_k,b_k]$ and $F''(x)>r>0$ (so that $F'(x)$ is increasing) since the proofs of the other cases are essentially the same. So, since $G(x)$ is monotonically increasing on $[a_k,b_k]$, we have by the second mean value theorem (Lemma \ref{secondmvlem}) that there exists $c_k\in[a_k,b_k)$ such that
    \begin{equation}\label{akbksecond}
        \left|\int_{a_k}^{b_k}G(x)e^{iF(x)}\mathrm{d}x\right|\leq M\left|\int_{c_k}^{b_k}e^{iF(x)}\mathrm{d}x\right|.
    \end{equation}
    Then, since $F'(x)$ is increasing on $[c_k,b_k]$ it has at most one zero in this interval, which we denote by $d_k$. If $F'(x)$ does not vanish on $[c_k,b_k]$ then we just let $d_k$ be the midpoint $d_k=(b_k+c_k)/2$. We now write
    \begin{equation}\label{I123eq}
        \int_{c_k}^{b_k}e^{iF(x)}\mathrm{d}x=\int_{c_k}^{d_k-\delta}+\int_{d_k-\delta}^{d_k+\delta}+\int_{d_k+\delta}^{b_k}=I_1+I_2+I_3,
    \end{equation}
    where $\delta>0$ is a small parameter which we choose later, and we assume that $d_k-\delta>c_k$ and $d_k+\delta<b_k$ for otherwise we can just remove $I_1$ or $I_3$ from \eqref{I123eq}. First we note that
    \begin{equation*}
        |I_2|\leq 2\delta
    \end{equation*}
    since $|e^{iF(x)}|=1$. Now, since $F'(x)$ is increasing, we have for any $x\in[d_k+\delta,b_k]$
    \begin{equation*}
        F'(x)\geq F'(d_k+\delta)=\int_{d_k}^{d_k+\delta}F''(t)\mathrm{d}t\geq r\delta.
    \end{equation*}
    Thus, Lemma \ref{mvlem} gives
    \begin{equation*}
        |I_3|\leq\frac{4}{r\delta}
    \end{equation*}
    and by analogous reasoning one obtains the same bound for $|I_1|$. Substituting these bounds into \eqref{I123eq} and \eqref{akbksecond} one has
    \begin{equation*}
        \left|\int_{a_k}^{b_k}G(x)e^{iF(x)}\mathrm{d}x\right|\leq\frac{8M}{r\delta}+2M\delta.
    \end{equation*}
    The desired result then follows upon setting $\delta=2/\sqrt{r}$ and summing over all ${k\in\{1,\ldots,K\}}$.
\end{proof}

\subsection{Complete exponential sum estimates}\label{exponsub}
When applying the van der Corput method, we will need to estimate complete exponential sums of the following form:
\begin{align}\label{K2v1}
K_2(a_1,a_2,a_3,q,c):=\frac{1}{c^{1/2}}\sum_{\substack{1\leq x\leq c\\(x(x+q),c)=1}}e\left(\frac{a_1x+a_2\overline{x}^2+a_3\overline{(x+q)}^2}{c}\right)
\end{align}
where each argument of $K_2$ is an integer with $q,c>0$, and $\overline{x}$ denotes the multiplicative inverse of $x$ modulo $c$. When $q=0$ we simplify notation:
\begin{align}\label{K2v2}
K_2(a_1,a_2,c):=\frac{1}{c^{1/2}}\sum_{\substack{1\leq x\leq c\\(x,c)=1}}e\left(\frac{a_1x+a_2\overline{x}^2}{c}\right).
\end{align}
Such sums may be referred to as ``quadratic Kloosterman sums". To bound \eqref{K2v1} and \eqref{K2v2}, or averages thereof, we begin with the following lemma which is based on the method of $p$-adic stationary phase. We remark that a similar approach was used in the work of Mangerel~\cite{mangerel2021}, wherein estimates for sums of the form \eqref{K2v2} were used to study square-free numbers in arithmetic progressions to smooth moduli. 
\begin{lemma}
\label{lem:KK22}
Let $c$ be an integer with prime factorization 
$$c=p_1^{\alpha_1}\dots p_k^{\alpha_k}.$$
 For each $1\le j \le k$ write
\begin{equation*}
    \alpha_{j,1}+\alpha_{j,2}=\alpha_j,
\end{equation*}
where $\alpha_{j,1}=\lfloor{\alpha_{j}/2}\rfloor$ and $\alpha_{j,2}=\lceil{\alpha_j/2}\rceil$. Define
$$c_1=\prod_{j=1}^{k}p_j^{\alpha_{j,1}}, \quad c_2=\prod_{j=1}^{k}p_j^{\alpha_{j,2}}.$$
Then 
\begin{align}
&K_2(a_1,a_2,a_3,q,c)\notag\\
&\quad\ll\#\{w\!\!\!\mod{c_1}: (w(w+q),c)=1,\ a_1-2a_2\overline{w}^{3}-2a_3\overline{(w+q)}^{3}\equiv 0\!\!\mod{c_1}  \}\label{K2simpeq}
\end{align}
with the implied constant depending on the primes $p_1,\ldots,p_k$ but not $\alpha_1,\ldots,\alpha_k$.
\end{lemma}
\begin{proof}
We first note that any $1\leq x\leq c$ can be uniquely written as $x=w+zc_2$ with $1\leq w\leq c_2$ and $0\leq z< c_1$. With this notation, since $p_1\cdots p_k\mid c_2$, it follows that $(x(x+q),c)=1$ if and only if $(w(w+q),c)=1$. Hence,
\begin{align*}
K_2(a_1,a_2,a_3,q,c)=\frac{1}{c^{1/2}}\sum_{\substack{z\ \text{mod}\ {c_1} \\ w\ \text{mod}\ {c_2} \\ (w(w+q),c)=1}}e\left(\frac{F(w+zc_2)}{c}\right)
\end{align*}
where 
\begin{align*}
F(x)=a_1x+a_2x^{\phi(c)-2}+a_3(x+q)^{\phi(c)-2},
\end{align*}
with $\phi$ the Euler totient function. Now, since $2\alpha_{j,2}\geq \alpha_j$ for each $1\le j \le k$, we have $c_2^2\equiv 0$ (mod $c$) so that for any $m\geq 0$,
\begin{equation*}
    (w+zc_2)^m\equiv w^m+zc_2mw^{m-1}\pmod{c}.
\end{equation*}
In particular, 
\begin{equation*}
    F(w+zc_2)\equiv F(w)+zc_2F'(w)\pmod{c}
\end{equation*}
and thus
\begin{align}
    K_2(a_1,a_2,a_3,q,c)&=\frac{1}{c^{1/2}}\sum_{\substack{z\ \text{mod}\ {c_1} \\ w\ \text{mod}\ {c_2} \\ (w(w+q),c)=1}}e\left(\frac{F(w)}{c}+\frac{zF'(w)}{c_1}\right) \notag\\
    &=\frac{1}{c^{1/2}}\sum_{\substack{w\ \text{mod}\ {c_2} \\ (w(w+q),c)=1}}e\left(\frac{F(w)}{c}\right)\sum_{z\ \text{mod}\ {c_1}}e\left(\frac{zF'(w)}{c_1}\right)\notag\\
    & \ll \#\{w\mod{c_2} \ : \  (w(w+q),c)=1,\ F'(w)\equiv 0 \mod{c_1} \} \label{N2set},
\end{align}
where in the last line we used that $c_1\leq\sqrt{c}$ by definition. Now, since each $\alpha_{j,2}$ is equal to $\alpha_{j,1}$ or $\alpha_{j,1}+1$, any $w$ with $1\leq w\leq c_2$ can be written as $w=w_0+w_1c_1$, where $1\leq w_0\leq c_1$ and $0\le w_1< p_1\cdots p_k$. Moreover, with this notation,
\begin{equation*}
    F'(w)\equiv F'(w_0)\pmod{c_1}.
\end{equation*}
As a result,
\begin{align}
    &K_2(a_1,a_2,a_3,q,c)\notag\\
    &\qquad\ll \#\{w\mod{c_1} \ : \  (w(w+q),c)=1,\ F'(w)\equiv 0 \mod{c_1} \}\label{N1set}
\end{align}
with the implied constant depending on $p_1,\dots,p_k$.
 This completes the proof of the lemma, since the fact that $c_1\mid \phi(c)$ gives
\begin{align*}
F'(w)&\equiv a_1-2a_2\overline{w}^{3}-2a_3\overline{(w+q)}^{3} \mod{c_1}.\qedhere
\end{align*}
\end{proof}

Using Lemma~\ref{lem:KK22} in conjunction with a bound on the number of cubic residues (Lemma \ref{lem:cubres}), we now give two bounds for $K_2$ in the forms required for our application. First, we give a simple pointwise bound needed in the proof of Proposition~\ref{powerprop2}.

\begin{lemma}
\label{lem:kman}
Suppose $a,m\in\mathbb{Z}$, $b\geq 2$ and $N\geq 1$. Then for any divisor $c$ of $b^{N}$ with $(a,c)=1$, we have 
\begin{align*}
    K_2(m,a,c)=O_b(1).
\end{align*}
\end{lemma}
\begin{proof}

Applying Lemma~\ref{lem:KK22} with the notation therein,
\begin{align}\label{K2bound1}
|K_2(m,a,c)|&\ll_b  \#\{w\!\!\!\mod{c_1}: (w,c)=1,\ m\equiv 2a \overline{w}^{3}\!\!\!\ \mod{c_1}  \}.
\end{align}
If $c_1$ is odd, then this reduces further to
\begin{equation*}
    |K_2(m,a,c)|\ll_b  \#\{w\!\!\!\mod{c_1}: (w,c)=1,\ m\overline{2a}\equiv  \overline{w}^{3}\!\!\!\ \mod{c_1}  \}
\end{equation*}
and the desired result follows from Lemma \ref{lem:cubres}. Now suppose $c_1$ is even. If $m$ is odd, then the right-hand side of \eqref{K2bound1} is zero. Then, if $m$ is even, \eqref{K2bound1} can be bounded further as
\begin{align*}
|K_2(m,a,c)|&\ll_b \#\{w\!\!\!\mod{c_1}: (w,c)=1,\ (m/2)\overline{a}\equiv  \overline{w}^{3}\!\!\!\ \mod{c_1/2}  \}
\end{align*}
and the desired result again follows from Lemma \ref{lem:cubres}.
\end{proof}

Next, we give an averaged bound required for Proposition~\ref{powerprop3} which arises after we apply the $A$-process (Lemma \ref{lem:vdc}).

\begin{lemma}
\label{lem:b-adic-average}
Let $b\geq 2$, $a,q\in\mathbb{Z}$ and $K\geq 1$. Then, for any $M\geq 1$,
\begin{equation*}
    \sum_{|m|\leq M}|K_2(m,a,-a,q,b^K)|\ll_b b^{K/2}+M.
\end{equation*}
\end{lemma}
\begin{proof}
    Let $c=b^K$. Applying Lemma~\ref{lem:KK22} with the notation therein,
    \begin{align*}
        &\sum_{|m|\leq M}|K_2(m,a,-a,q,c)|\\
        &\qquad\ll_b\#\left\{(w,m):(w(w+q),c)=1,\:|m|\leq M,\:m\equiv2a\left(\overline{w}^3-\overline{(w+q)}^3\right)\!\!\mod{c_1}\right\}\\
        &\qquad=\sum_{\substack{w\ \text{mod}\ c_1\\(w(w+q),c)=1}}\#\left\{|m|\leq M:m\equiv 2a\left(\overline{w}^3-\overline{(w+q)}^3\right)\!\!\mod{c_1}\right\}.
    \end{align*}
    Here, the summand counts $|m|\leq M$ in a single residue class mod $c_1$, so that
    \begin{equation*}
        \sum_{|m|\leq M}|K_2(m,a,-a,q,c)|\ll_b\sum_{\substack{w\ \text{mod}\ c_1\\(w(w+q),c)=1}}\left(\frac{2M}{c_1}+1\right)\ll M+c_1.
    \end{equation*}
    The desired result then follows upon noting that $c_1\leq \sqrt{c}=b^{K/2}$.
\end{proof}

\section{Proof of Proposition~\ref{powerprop2}}\label{Prop2sect}
Having established the preliminary results in Section \ref{Prelimsect}, we are now able to prove Proposition \ref{powerprop2}. The general idea is to build on the proof of Proposition \ref{powerprop1}, applying Poisson summation (Lemma \ref{lem:ps}) and relevant estimates for oscillatory integrals and exponential sums from Sections \ref{osssub} and \ref{exponsub}.

\begin{proof}[Proof of Proposition \ref{powerprop2}]
    Let $\varepsilon>0$. As in the proof of Proposition \ref{powerprop1}, it suffices to consider palindromes of a fixed length $N$, showing that
    \begin{equation*}
        \#\{n\in\Pi^*_b(N):\exists d\sim D\ \text{s.t.}\ d^2\mid n\}\ll_{b,\varepsilon}\frac{b^{2N/3}}{D^{2/3-\varepsilon}}
    \end{equation*}
    for $b^{N/4}\leq D\leq b^{2N/5}$. In what follows, we let $K$ be the largest integer such that ${b^K\leq b^{N/3}/D^{1/3+\varepsilon}}$, so that in particular,
    \begin{equation}\label{bKeq2}
        b^K\asymp_b \frac{b^{N/3}}{D^{1/3+\varepsilon}}.       
    \end{equation} 
    
    Following the start of the proof of Proposition \ref{powerprop1}, we have
    \begin{align}\label{elembound}
        &\#\{n\in\Pi_b^*(N):\exists d\sim D\ \text{s.t.}\ d^2\mid n\}\notag\\
        &\leq b^K\#\left\{d\sim D,\ s\in\left[\frac{b^{N-1}}{4D^2},\frac{b^N}{D^2}\right]: d^2s\equiv a\ \text{(mod}\ b^K\text{)},\ (sd,b)=1,\ \left|d^2s-\alpha \right|\leq \frac{b^{N}}{b^K}\right\}
    \end{align}
    for some integer $a$ with $0\leq a<b^K$ that does not end in a $0$ in base $b$ and $\alpha$ the $N$-digit number beginning with the mirror image of $a$ followed by a suitable number of trailing zeros. Since we are only considering palindromes in $\Pi^*_b(N)$ we may also assume that $a$ is coprime to $b$.
    
    We now introduce two smoothing functions into our argument, with valid choices plotted in Figure \ref{fig:Desmossmooth} for clarity. First, let $\psi$ be a smooth, piecewise monotonic function compactly supported on $[1/2,5/2]$ with $\psi(x)=1$ for all $x\in [1,2]$. Second, let $\phi$ be a smooth, piecewise monotonic function compactly supported on $[-2,2]$ with $\phi(x)=1$ for all $x\in[-1,1]$.
    \begin{figure}[t]
        \centering
        \includegraphics[width=0.6\textwidth]{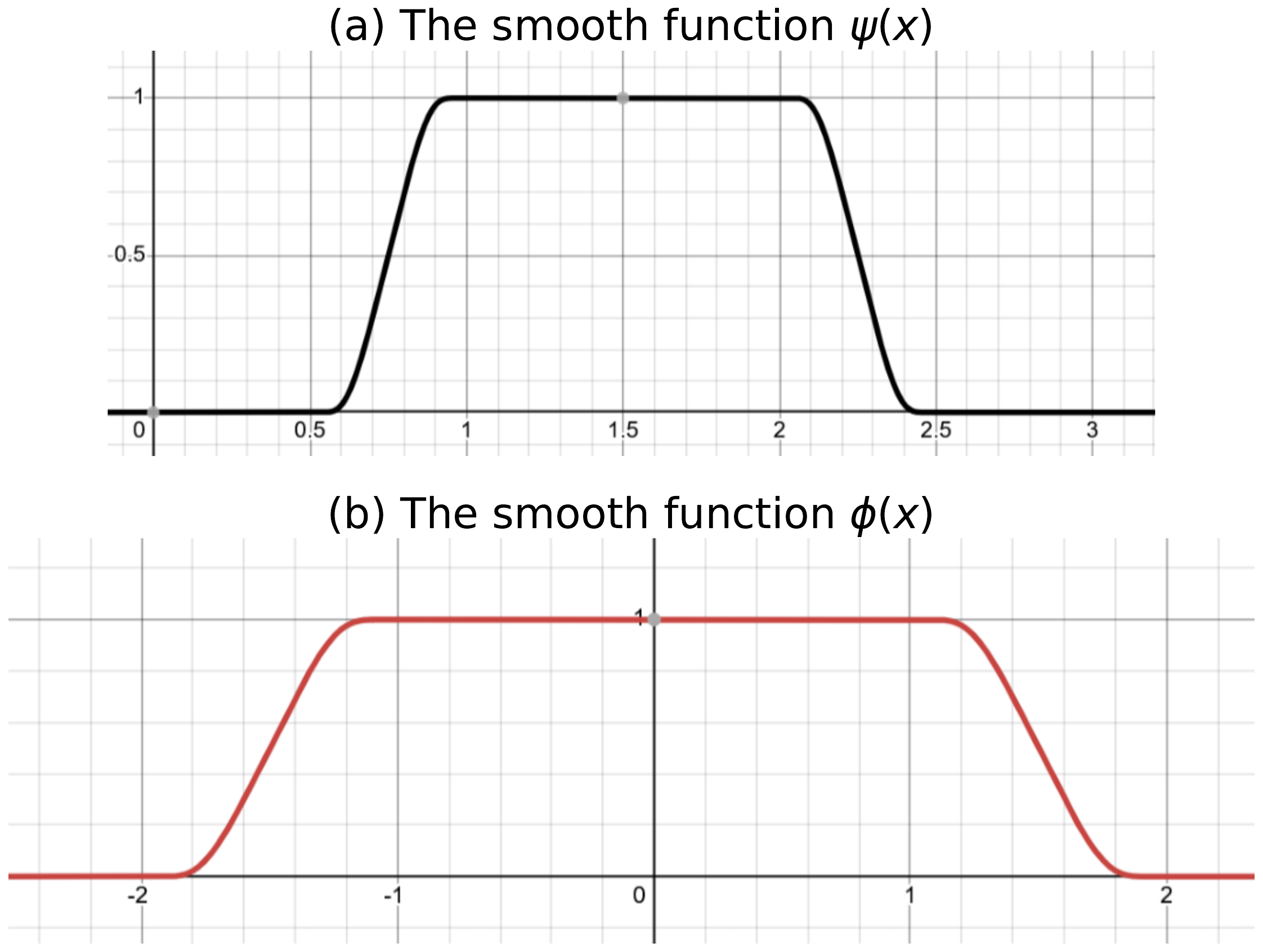}
        \caption{Example choices of the functions $\psi(x)$ and $\phi(x)$, plotted using Desmos \cite{Desmos}. Importantly, each function is smooth, compactly supported, piecewise monotonic and bounded.}\label{fig:Desmossmooth}
    \end{figure}
    
    Using these smoothing functions, we can bound \eqref{elembound} further as
    \begin{align}\label{preposseq}
        &\#\{n\in\Pi_b^*(N):\exists d\sim D\ \text{s.t.}\ d^2\mid n\}\notag\\
        &\qquad\qquad\qquad\leq b^K\sum_{\substack{d\in\Z \\ (d,b)=1}}\psi\left(\frac{d}{D}\right)\sum_{\substack{s\in\mathbb{Z}\\ s\equiv a\overline{d}^2\ \text{mod}\ b^K}}\phi((s-\alpha/d^2)D^2 b^{K-N}),
    \end{align}
    where $\overline{d}$ denotes a fixed representative of the residue class of the multiplicative inverse of $d$ mod $b^K$. Applying Poisson summation (Lemma~\ref{lem:ps}),
    \begin{align*}
        &\sum_{\substack{s\in\mathbb{Z}\\ s\equiv a\overline{d}^2\ \text{mod}\ b^K}}\phi((s-\alpha/d^2)D^2 b^{K-N})= \\ & \quad \quad \quad \quad \quad \frac{1}{b^K}\sum_{\ell\in\Z}e\left(\frac{a\overline{d}^2\ell}{b^K}\right)\int_{\R}\phi((u-\alpha/d^2)D^2b^{K-N})e\left(-\frac{u\ell}{b^K}\right)\mathrm{d}u.
    \end{align*}
 Performing a linear change of variables in the above integral then gives
    \begin{equation*}
        \int_{\R}\phi((u-\alpha/d^2)D^2b^{K-N})e\left(-\frac{u\ell}{b^K}\right)\mathrm{d}u=\frac{b^{N}}{D^2b^K}e\left(-\frac{\alpha\ell}{d^2b^K}\right)\widehat{\phi}\left(\frac{\ell b^N}{D^2b^{2K}}\right).
    \end{equation*}
    Substituting back into \eqref{preposseq}, we have
    \begin{align}\label{postposseq}
        &\#\{n\in\Pi_b^*(N):\exists d\sim D\ \text{s.t.}\ d^2\mid n\}\notag\\
        &\qquad\qquad\ll\frac{b^{N}}{D^2b^K}\sum_{\substack{d\in\Z\\(d,b)=1}}\psi\left(\frac{d}{D}\right)\sum_{\ell\in\Z}e\left(-\frac{\alpha\ell}{d^2b^K}\right)e\left(\frac{a\overline{d}^2\ell}{b^K}\right)\widehat{\phi}\left(\frac{\ell b^N}{D^2b^{2K}}\right).
    \end{align}
    Consider the Fourier transform $\widehat{\phi}$. Applying integration by parts $k\geq 1$ times, one has for any $x\in\R\setminus\{0\}$
    \begin{equation}\label{intpartseq}
        \widehat{\phi}(x)=\int_{\R}\phi(y)e(-xy)\mathrm{d}y=\left(\frac{1}{2\pi i}\right)^k\frac{1}{x^k}\int_{\R}\phi^{(k)}(y)e(-xy)\mathrm{d}y\ll\frac{1}{|x|^k},
    \end{equation}
    noting that $\phi^{(k)}(y)$ is bounded on $\mathbb{R}$ since $\phi$ is smooth and compactly supported in $[-2,2]$. This implies that
    \begin{align*}
        \sum_{|\ell|> D^2b^{2K}/b^{N(1-\varepsilon)}}e\left(-\frac{\alpha \ell}{d^2b^K}\right)e\left(\frac{a\overline{d}^2\ell}{b^{K}}\right)\widehat{\phi}\left(\frac{\ell b^N}{D^2b^{2K}}\right)&\ll \sum_{|\ell|> D^2b^{2K}/b^{N(1-\varepsilon)}}\left|\widehat{\phi}\left(\frac{\ell b^N}{D^2b^{2K}}\right)\right| \\ 
        &\ll \sum_{|\ell|> D^2b^{2K}/b^{N(1-\varepsilon)}}\left(\frac{D^2b^{2K}}{|\ell| b^N}\right)^{k} \\ & \ll\frac{D^2 b^{2K}}{b^N}b^{N\varepsilon(1-k)}\\
        &=O(b^{-1000N})
    \end{align*}
after taking $k$ sufficiently large, depending only on $\varepsilon$.
    Hence, writing
    \begin{equation*}
        \sum_{\ell\in\Z}=\sum_{\ell=0}+\sum_{1\leq|\ell|\leq D^2b^{2K}/b^{N(1-\varepsilon)}}+\sum_{|\ell|> D^2b^{2K}/b^{N(1-\varepsilon)}}
    \end{equation*}
    we have that \eqref{postposseq} is further bounded as
    \begin{align}\label{mainlsumeq}
        &\#\{n\in\Pi_b^*(N):\exists d\sim D\ \text{s.t.}\ d^2\mid n\}\ll\frac{b^N}{Db^K}+\frac{b^N}{D^2b^K}\sum_{1\leq|\ell|\leq D^2b^{2K}/b^{N(1-\varepsilon)}}\left|S(\ell)\right|
    \end{align}
where 
\begin{align}
\label{eq:sldef}
S(\ell)=\sum_{\substack{d\in\Z\\(d,b)=1}}\psi\left(\frac{d}{D}\right)e\left(-\frac{\alpha\ell}{d^2b^K}\right)e\left(\frac{a\overline{d}^{2}\ell}{b^K}\right).
\end{align}
We now partition the sum in \eqref{mainlsumeq} according to the greatest common divisor of $b^{K}$ and $\ell$:
\begin{align}\label{SLparteq}
    \sum_{1\leq|\ell|\leq D^2b^{2K}/b^{N(1-\varepsilon)}}\left|S(\ell)\right|=\sum_{\substack{c|b^{K}}}\tilde{S}(c)
\end{align}
where 
\begin{align}
\label{eq:Sjdef}
\tilde{S}(c)=\sum_{\substack{1\leq|\ell|\leq D^2b^{2K}c/b^{N(1-\varepsilon)+K} \\ (\ell,c)=1}}\left|S(b^{K}\ell/c)\right|.
\end{align}
The condition $(\ell,c)=1$ is important here, as it will ultimately allow us to apply Lemma \ref{lem:kman}.
Now, consider
\begin{align*}
    S(b^{K}\ell/c)=\sum_{\substack{d\in\Z\\(d,b)=1}}\psi\left(\frac{d}{D}\right)e\left(-\frac{\alpha\ell}{d^2c}\right)e\left(\frac{a\overline{d}^{2}\ell}{c}\right).
\end{align*}

By Lemma~\ref{lem:ps}
\begin{align}
S(b^{K}\ell/c)&=\frac{1}{c^{1/2}}\sum_{m\in \Z}K_2(m,a\ell,c)\int_{\R}\psi\left(\frac{u}{D}\right)e\left(-\frac{\alpha\ell}{u^2c}-\frac{mu}{c}\right)\mathrm{d}u \notag\\ 
&=\frac{D}{c^{1/2}}\sum_{m\in \Z}K_2(m,a\ell,c)\int_{\R}\psi\left(u\right)e\left(-\frac{\alpha\ell}{D^2u^2c}-\frac{mDu}{c}\right)\mathrm{d}u.\label{Sbkleq}
\end{align}

    We next truncate summation by considering the decay of the integral in \eqref{Sbkleq}. For fixed $m$, we write
    \begin{equation*}
        F(u)=-\frac{\alpha\ell}{u^2D^2c}-\frac{mDu}{c}
    \end{equation*}
    for the phase of the exponential in \eqref{Sbkleq}. Here, we restrict to $u\in[1/2,5/2]$ since this is the support of $\psi$.
    Now,
    \begin{align}\label{Fpeq}
        F'(u)&=\frac{2\alpha\ell}{u^3D^2c}-\frac{mD}{c},\\
        |F^{(k)}(u)|&=\frac{(k+1)!\thinspace\alpha|\ell|}{u^{k+2}D^2c}\qquad k\geq 2.\label{Fkeq}
    \end{align}
    If $|m|>\alpha|\ell|/D^{3-\varepsilon}$ then the second term in \eqref{Fpeq} dominates. That is,
    \begin{equation*}
        |F'(u)|\asymp_{\varepsilon}\frac{|m|D}{c},\qquad\text{for}\  |m|>\frac{\alpha|\ell|}{D^{3-\varepsilon}}.
    \end{equation*}
    Equipped with these derivative calculations, we may now apply Lemma~\ref{lem:nostat} to bound the oscillatory integral in~\eqref{Sbkleq}. Here, we remark that the derivatives of $\psi$ have suitable decay via a repeated application of integration by parts (cf.~\eqref{intpartseq}). So, by Lemma~\ref{lem:nostat}, we have for any integer $k\ge 2$,
    \begin{equation*}
        \left|\int_{\R} \psi(u)e\left(-\frac{\alpha \ell}{u^2D^2c}-\frac{mDu}{c}\right)\mathrm{d}u\right|\ll_{k,\varepsilon}\left(\frac{c}{|m|D}\right)^k \quad \text{if} \quad |m|>\frac{\alpha|\ell|}{D^{3-\varepsilon}}.
    \end{equation*}
    Since $c\le b^{K}$, we get 
    \begin{align}\label{largereq1}
        \sum_{|m|>\alpha|\ell|/D^{3-\varepsilon}}\left|\int_{\R} \psi(u)e\left(-\frac{\alpha \ell}{u^2D^2c}-\frac{mDu}{c}\right)\mathrm{d}u\right|&\ll_{k,\varepsilon}\frac{b^{Kk}}{D^k}\sum_{|m|>\alpha|\ell|/D^{3-\varepsilon}}\frac{1}{|m|^k}\notag\\
        &\ll_k\frac{b^{Kk}D^{(3-\varepsilon)(k-1)}}{D^{k}\alpha^{k-1}|\ell|^{k-1}}.
    \end{align}
     Recalling that $|\ell|\geq 1$, $\alpha\asymp_b b^N$, $D\leq b^{0.4N}$ and $b^K\leq b^{N/3}/D^{1/3}$, we have that \eqref{largereq1} can be further bounded as
    \begin{equation}\label{largereq2}
        \sum_{|m|>\alpha|\ell|/D^{3-\varepsilon}}\left|\int_{\R} \psi(u)e\left(-\frac{\alpha \ell}{u^2D^2c}-\frac{mDu}{c}\right)\mathrm{d}u\right|\ll_{k,\varepsilon,b}\frac{1}{b^{N(\frac{1}{5}+\varepsilon(k-1))}}
    \end{equation}
    after a suitable rescaling of $\varepsilon$. Hence, upon taking $k$ to be sufficiently large,
    \begin{align}\label{Sbklbigeq}
        &S(b^{K}\ell/c)\notag\\
        &=\frac{D}{c^{1/2}}\sum_{|m|\le \alpha|\ell|/D^{3-\varepsilon}}K_2(m,a\ell,c)\int_{\R}\psi\left(u\right)e\left(-\frac{\alpha\ell}{u^2D^2c}-\frac{mDu}{c}\right)\mathrm{d}u+O_{b,\varepsilon}(b^{-100N}).
    \end{align}
    By~\eqref{Fkeq} and Lemma~\ref{mvlem2}
    \begin{align*}
        \int_{\R}\psi\left(u\right)e\left(-\frac{\alpha\ell}{u^2D^2c}-\frac{mDu}{c}\right)\mathrm{d}u\ll \frac{Dc^{1/2}}{(\alpha|\ell|)^{1/2}},
    \end{align*}
    which, substituted into \eqref{Sbklbigeq}, implies 
    \begin{align*}
        S(b^{K}\ell/c)&\ll\frac{D^2}{(\alpha |\ell|)^{1/2}}\sum_{|m|\le \alpha|\ell|/D^{3-\varepsilon}}|K_2(m,a\ell,c)|+O_{b,\varepsilon}(b^{-100N}).
    \end{align*}
    Recalling that $(a\ell,c)=1$ (since $(a,b)=1$ and $(\ell,c)=1$), Lemma~\ref{lem:kman} gives $K_2(m,a\ell,c)=O_b(1)$ and thus
    \begin{align}\label{Sbklshorteq}
        S(b^{K}\ell/c)\ll_{b,\varepsilon} \frac{(\alpha |\ell|)^{1/2}}{D^{1-\varepsilon}}.
    \end{align}
    Substituting \eqref{Sbklshorteq} back into \eqref{eq:Sjdef} and \eqref{SLparteq}, and noting that $\alpha\asymp_b b^N$ then gives
    \begin{equation}\label{finalSleq}
        \sum_{1\leq|\ell|\leq D^2b^{2K}/b^{N(1-\varepsilon)}}\left|S(\ell)\right|\ll_{b,\varepsilon} d(b^K)\sum_{1\leq |\ell|\leq D^2b^{2K}/b^{N(1-\varepsilon)}}\frac{(\alpha|\ell|)^{1/2}}{D^{1-\varepsilon}}\ll_{b,\varepsilon} \frac{D^{2+\varepsilon}b^{(3+\varepsilon)K}}{b^{N(1-\frac{3}{2}\varepsilon)}},
    \end{equation}
    where $d(b^K)$ is the number of divisors of $b^K$, for which one has the standard bound $d(b^K)\ll_\varepsilon b^{\varepsilon K}$. 
    Finally, substituting~\eqref{finalSleq} into~\eqref{mainlsumeq} and suitably rescaling $\varepsilon$ yields 
    \begin{align*}
        \#\{n\in\Pi_b^*(N):\exists d\sim D\ \text{s.t.}\ d^2\mid n\}&\ll_{b,\varepsilon}\frac{b^{N}}{Db^K}+D^{3\varepsilon} b^{2K}. 
    \end{align*}
    Recalling the choice~\eqref{bKeq2} of $b^K$, we complete the proof.
\end{proof}

\begin{remark}
    The start of the above proof (up to \eqref{mainlsumeq}) may also be obtained via a smoothed-variant of the Erd\H{o}s-Tur\'an inequality. See for example \cite[p.~8]{mont94} for the usual ``unsmoothed" version of this inequality.
\end{remark}

\section{Proof of Proposition~\ref{powerprop3}}\label{Prop3sect}
In this final section, we prove Proposition \ref{powerprop3}. As discussed in Section \ref{Outlinesect}, upon combination with Propositions \ref{powerprop1} and \ref{powerprop2}, this allows us to deduce our main result, Theorem \ref{asymthm}.

The proof of Proposition \ref{powerprop3} begins in the same way as that of Proposition~\ref{powerprop2}. However, this time we also apply our $A$-process (Lemma \ref{lem:vdc}), which requires different exponential sum estimates and notably allows us to attain a non-trivial bound on $S_b(x,D)$ for values of $D$ less than $x^{1/4}$.

\begin{proof}[Proof of Proposition \ref{powerprop3}]
Let $\varepsilon>0$. As in the proofs of Propositions \ref{powerprop1} and \ref{powerprop2}, it suffices to consider palindromes of a fixed length $N$, showing that
    \begin{equation*}
        \#\{n\in\Pi^*_b(N):\exists d\sim D\ \text{s.t.}\ d^2\mid n\}\ll_{b,\varepsilon}\frac{b^{3N/5}}{D^{9/20-\varepsilon}}
    \end{equation*}
    for $b^{2N/9}\leq D\leq b^{N(8/31-\varepsilon)}$. In what follows, we let $K$ be the largest integer such that ${b^K\leq b^{2N/5}/D^{11/20+\varepsilon}}$, so that in particular,
    \begin{equation}\label{bKeq3}
        b^K\asymp_b \frac{b^{2N/5}}{D^{11/20+\varepsilon}}.       
    \end{equation} 
    We now argue as in the proof of Proposition~\ref{powerprop2} with the same notation up until~\eqref{mainlsumeq} and without the smoothing factor $\psi(d/D)$. This gives
    \begin{align}\label{mainlsumeq1}
        &\#\{n\in\Pi_b^*(N):\exists d\sim D\ \text{s.t.}\ d^2\mid n\}\ll\frac{b^N}{Db^K}+\frac{b^N}{D^2b^K}\sum_{1\leq|\ell|\leq D^2b^{2K}/b^{N(1-\varepsilon)}}\left|S(\ell)\right|,
    \end{align}
where 
\begin{align}
\label{eq:sldef1}
S(\ell)=\sum_{\substack{d\sim D\\(d,b)=1}}e\left(-\frac{\alpha\ell}{d^2b^K}\right)e\left(\frac{a\overline{d}^{2}\ell}{b^K}\right).
\end{align}
Next, diverging from the proof of Proposition \ref{powerprop2}, we do not partition $S(\ell)$ any further and instead apply Lemma~\ref{lem:vdc} with 
\begin{align}
\label{eq:Qdef1}
    Q=D^{1/2}.
\end{align}
This gives
\begin{align}
\label{eq:abc}
S(\ell)\ll D^{3/4}+D^{1/4}\left(\sum_{q\le D^{1/2}}|\tilde{S}(\ell,q)|\right)^{1/2}
\end{align}
where
\begin{align*}
\tilde{S}(\ell,q)=\sum_{\substack{d\in \Z}}\psi\left(\frac{d}{D}\right)e\left(-\frac{\alpha\ell}{b^K}\left(\frac{1}{d^2}-\frac{1}{(d+q)^2}\right)\right)e\left(\frac{a\ell(\overline{d}^{2}-\overline{(d+q)}^{2})}{b^K}\right)
\end{align*}
with $\psi(d/D)$ the same smoothing function from the proof of Proposition \ref{powerprop2} (see Figure \ref{fig:Desmossmooth}). We now apply Lemma~\ref{lem:ps} (Poisson summation) to get
\begin{align}
&\tilde{S}(\ell,q)\notag\\
&=\frac{1}{b^{K/2}}\sum_{m\in \Z}K_2(m,a\ell,-a\ell,q,b^K)
\int_{\R}\psi\left(\frac{u}{D}\right)e\left(-\frac{\alpha \ell}{b^K}\left(\frac{1}{u^2}-\frac{1}{(u+q)^2}\right)-\frac{mu}{b^K} \right)\mathrm{d}u\notag\\ 
&=\frac{D}{b^{K/2}}\sum_{m\in \Z}K_2(m,a\ell,-a\ell,q,b^K)
\int_{\R}\psi\left(u\right)e\left(-\frac{\alpha \ell}{b^K}\left(\frac{1}{D^2u^2}-\frac{1}{(Du+q)^2}\right)-\frac{mDu}{b^K} \right)\mathrm{d}u. \label{slqeq}
\end{align}
With a view to applying Lemma \ref{lem:nostat}, we let
\begin{equation*}
    F(u)=-\frac{\alpha \ell}{b^K}\left(\frac{1}{D^2u^2}-\frac{1}{(Du+q)^2}\right)-\frac{mDu}{b^K} 
\end{equation*}
denote the phase of the exponential in \eqref{slqeq}, with $u\in[1/2,5/2]$ as this is the support of $\psi$. We have
\begin{align}
        F'(u)&=\frac{2\alpha\ell}{b^K}\left(\frac{1}{D^2u^3}-\frac{D}{(Du+q)^3}\right)-\frac{mD}{b^K}\notag\\
        &=\frac{2\alpha\ell}{b^K}\left(\frac{3D^2u^2q+3Duq^2+q^3}{D^2u^3(Du+q)^3}\right)-\frac{mD}{b^K},\label{Fpeq2}\\
        |F^{(k)}(u)|&=(k+1)!\frac{\alpha|\ell|}{b^K}\left(\frac{1}{D^2u^{k+2}}-\frac{D^k}{(Du+q)^{k+2}}\right),\qquad k\geq 2.\label{Fkeq2}
    \end{align}
    Since $q\leq D^{1/2}$, we see that the second term in \eqref{Fpeq2} dominates for $|m|>\alpha|\ell|q/D^{4-\varepsilon}$. Then, since $\alpha \asymp_b b^{N}$ and 
    $$1\leq|\ell|\leq D^2b^{2K}/b^{N(1-\varepsilon)}$$ 
 this implies that
    \begin{equation}\label{fpuprop4}
        |F'(u)|\asymp_{\varepsilon}\frac{|m|D}{b^K},\qquad\text{for}\quad   |m|>\frac{b^{2K}q}{D^{2-\varepsilon}}b^{N\varepsilon}.
    \end{equation}
    As we are only considering $D$ with $\log D\asymp_b N$, we may suitably rescale $\varepsilon$ to recast and simplify \eqref{fpuprop4} as
    \begin{equation*}
        |F'(u)|\asymp_{b,\varepsilon}\frac{|m|D}{b^K},\qquad\text{for}\quad   |m|>\frac{b^{2K}q}{D^{2-\varepsilon}}.
    \end{equation*}
    Thus, applying Lemma \ref{lem:nostat}, we have for any $k\geq 2$,
    \begin{align}\label{largemeq1}
        &\sum_{|m|>b^{2K}q/D^{2-\varepsilon}}\left|\int_{\R} \psi\left(u\right)e\left(-\frac{\alpha \ell}{b^K}\left(\frac{1}{D^2u^2}-\frac{1}{(Du+q)^2}\right)-\frac{mDu}{b^K} \right)\mathrm{d}u\right|\notag\\
        &\qquad\ll_{k,\varepsilon}\frac{b^{Kk}}{D^k}\sum_{|m|>b^{2K}q/D^{2-\varepsilon}}\frac{1}{|m|^k}\ll_k \frac{qb^{2K}}{D^{2-\varepsilon}}\left(\frac{D^{1-\varepsilon}}{b^{K} q}\right)^{k}\leq\frac{b^{2K}}{D^{3/2-\varepsilon}}\left(\frac{D^{1-\varepsilon}}{b^K}\right)^k,
    \end{align}
    where in the last step we used that $1\leq q\leq D^{1/2}$. By our choice \eqref{bKeq3} of $b^K$, and the fact that $b^{2N/9}\leq D\leq b^{N(8/31-\varepsilon)}$, we have
    $$\frac{D^{1-\varepsilon}}{b^{K}}\leq b^{-N\varepsilon}.$$
    Hence, taking $k$ to be sufficiently large in terms of $\varepsilon$,
    \begin{align}
    \label{eq:Slq}
        &\tilde{S}(\ell,q)=\frac{D}{b^{K/2}}\sum_{|m|\le \frac{b^{2K}q}{D^{2-\varepsilon}}}K_2(m,a\ell,-a\ell,q,b^K)\notag\\
        &\qquad\qquad\qquad\qquad\cdot\int_{\R}\psi\left(u\right)e\left(-\frac{\alpha \ell}{b^K}\left(\frac{1}{D^2u^2}-\frac{1}{(Du+q)^2}\right)-\frac{mDu}{b^K} \right)\mathrm{d}u\\
        &\qquad\qquad\qquad\qquad\qquad\qquad\qquad\qquad\qquad\qquad\qquad\qquad\qquad\quad+O_{b,\varepsilon}(b^{-100N}).\notag
    \end{align}
    Now, from \eqref{Fkeq2},
    \begin{align*}
        F''(u)\gg \frac{\alpha|\ell|q}{b^KD^3}, \quad u\in\left[\frac{1}{2},\frac{5}{2}\right].
    \end{align*}
    Thus, applying Lemma~\ref{mvlem2}, 
    \begin{align*}
        \int_{\R}\psi\left(u\right)e\left(-\frac{\alpha \ell}{b^K}\left(\frac{1}{D^2u^2}-\frac{1}{(Du+q)^2}\right)-\frac{mDu}{b^K} \right)\mathrm{d}u\ll \left(\frac{\alpha|\ell| q}{b^KD^3}\right)^{-1/2}
    \end{align*}
    which substituted into~\eqref{eq:Slq} implies that 
    \begin{align}\label{slqmsum}
        \tilde{S}(\ell,q)\ll_{b,\varepsilon} \frac{D^{5/2}}{(\alpha |\ell| q)^{1/2}}\sum_{|m|\le \frac{b^{2K}q}{D^{2-\varepsilon}}}|K_2(m,a\ell,-a\ell,q,b^K)|.
    \end{align}
    Substituting~\eqref{slqmsum} back into~\eqref{eq:abc} and applying Lemma~\ref{lem:b-adic-average},
    \begin{align*}
        S(\ell)&\ll_{b,\varepsilon} D^{3/4}+\frac{D^{3/2}}{(\alpha |\ell|)^{1/4}}\left(\sum_{q\le D^{1/2}}\frac{1}{q^{1/2}}\sum_{|m|\le \frac{b^{2K}q}{D^{2-\varepsilon}}}|K_2(m,a\ell,-a\ell,q,b^K)|\right)^{1/2} \\ 
        &\ll_b  D^{3/4}+\frac{D^{3/2}}{(\alpha |\ell|)^{1/4}}\left(\sum_{q\le D^{1/2}}\frac{1}{q^{1/2}}\left(b^{K/2}+\frac{b^{2K}q}{D^{2-\varepsilon}}\right)\right)^{1/2}\\
        &\ll  D^{3/4}+\frac{D^{3/2}}{(\alpha |\ell|)^{1/4}}\left(D^{1/4} b^{K/2}+\frac{b^{2K}}{D^{5/4-\varepsilon}}\right)^{1/2}.\\
    \end{align*}
    Here,
    \begin{equation*}
        D^{1/4} b^{K/2}\ll_b \frac{b^{2K}}{D^{5/4-\varepsilon}}
    \end{equation*}
    which follows from the definition~\eqref{bKeq3} of $b^K$ and the fact that $D\leq b^{8N/31-\varepsilon}$. Therefore,
    \begin{align}\label{Sbklfinal4}
        S(\ell)
        &\ll_{b,\varepsilon} D^{3/4}+\frac{b^{K}D^{7/8+\varepsilon/2}}{(\alpha |\ell|)^{1/4}}.
    \end{align}
    Summing over $|\ell|$ then yields
    \begin{align*}
        \sum_{1\leq|\ell|\leq D^2b^{2K}/b^{N(1-\varepsilon)}}\left|S(\ell)\right|&\ll_{b,\varepsilon}\sum_{1\leq|\ell|\leq D^2b^{2K}/b^{N(1-\varepsilon)}}\left(D^{3/4}+\frac{b^{K}D^{7/8+\varepsilon/2}}{(\alpha |\ell|)^{1/4}}\right)\\
        &\ll_{b,\varepsilon} \frac{D^{11/4}b^{2K}}{b^{N(1-\varepsilon)}}+\frac{b^{5K/2}D^{19/8+\varepsilon/2}}{b^{N(1-3\varepsilon/4)}},
    \end{align*}
    where we have used that $\alpha\asymp_b b^N$. From here, \eqref{mainlsumeq1} gives that
    \begin{align}\label{prop4penul}
        &\#\{n\in\Pi_b^*(N):\exists d\sim D\ \text{s.t.}\ d^2\mid n\}\ll_{b,\varepsilon}\frac{b^N}{Db^K}+b^{K}D^{3/4+\varepsilon}+b^{3K/2}D^{3/8+\varepsilon},
    \end{align}
    where we have suitably rescaled $\varepsilon$.
    Our choice of $b^{K}$ in~\eqref{bKeq3} essentially balances the first and third terms in~\eqref{prop4penul}, giving
    \begin{align}\label{prop4final}
        \#\{n\in\Pi_b^*(N):\exists d\sim D\ \text{s.t.}\ d^2\mid n\}& \ll_{b,\varepsilon} \frac{b^{3N/5}}{D^{9/20-\varepsilon}}+b^{2N/5}D^{1/5}.
    \end{align}
    Since $b^{2N/9}\leq D\leq b^{N(8/31-\varepsilon)}$, the second term in \eqref{prop4final} is of smaller order than the first, and can thus be removed. This completes the proof of the Proposition.
\end{proof}

\printbibliography
\end{document}